\documentclass[draft]{amsart}

\usepackage{pgf,tikz}
\usetikzlibrary{arrows}
\usepackage{cite}
\usepackage{stmaryrd}

\usepackage{amssymb,graphicx,enumerate,color}
\usepackage{multirow}

\usepackage{graphicx}

\newcommand{\A}{\mathcal{A}}

\newcommand{\D}{{\mathcal{D}}}

\newcommand{\F}{{\mathcal{F}}}

\newcommand{\beqs}{\begin{equation*}}
\newcommand{\eeqs}{\end{equation*}}
\numberwithin{equation}{section}
 \theoremstyle{plain}
\newtheorem{theorem}{Theorem}[section]

\newtheorem{corollary}[theorem]{Corollary}

\theoremstyle{remark}

\theoremstyle{empty}

\begin{document}

\makeatletter
\def\imod#1{\allowbreak\mkern10mu({\operator@font mod}\,\,#1)}
\makeatother


\author{Ali Kemal Uncu}
   \address{Austrian Academy of Sciences, Johann Radon Institute for Computational and Applied Mathematics, Altenbergerstrasse 69, 4040 Linz, Austria}
   \email{akuncu@risc.jku.at}

\title[On a weighted spin of the Lebesgue Identity]{On a weighted spin of the Lebesgue Identity}



\begin{abstract} Alladi studied partition theoretic implications of a two variable generalization of the Lebesgue identity. In this short note, we focus on a slight variation of the basic hypergeometric sum that Alladi studied. We present two new partition identities involving weights.
\end{abstract}
   


\thanks{Research of the author is supported by the Austrian Science Fund FWF, SFB50-09 and SFB50-11 Projects.}

\date{\today}
   
\maketitle

\section{Introduction}

One of the fundamental identities in the theory of partitions and $q$-series is the Lebesgue identity: \begin{equation}
\label{Lebesgue} \sum_{n\geq0} \frac{(-aq)_n}{(q)_n} q^{\frac{n(n+1)}{2}} = \frac{(-aq^2;q^2)_\infty}{(q;q^2)_\infty},
\end{equation} where $a$ and $q$ are variables and the $q$-Pochhammer symbol is defined as follows \[(a)_n := (a;q)_n := \prod_{i=0}^{n-1} (1-a q^i),\] for any $n\in \mathbb{Z}\cup\{\infty\}$. Some combinatorial implications of this result were studied by Alladi \cite{Alladi_Lebesgue}. In the same paper, he also did a partition theoretic study of a summation formula due to Ramanujan \cite[(1.3.13), p 13]{LostNotebook2} \begin{equation}\label{GeneralizedLebesgueIdentity}\sum_{n\geq 0} \frac{(-b/a)_n a^n q^{n(n+1)/2}}{(q)_n(bq)_n} = \frac{(-aq)_\infty}{(bq)_\infty}.\end{equation} Alladi called this identity and it's dilated forms \textit{Generalized Lebesgue identities}.

We would like to study a similar function that is not directly related to \eqref{GeneralizedLebesgueIdentity} or that satisfies a summation formula, but that still manifest beautiful relations. Let $a,\ z$ and $q$ be variables and define \begin{equation}\label{Ffunction}
\F(a,z,q) := \sum_{n\geq0} \frac{(za)_n}{(q)_n(zq)_n} z^n q^{\frac{n(n+1)}{2}}.
\end{equation} Looking at $\F(b,a,q)$ it is clear that this sum is ---so to speak--- a sibling of the Generalized Lebesgue identity \eqref{GeneralizedLebesgueIdentity}, and $\F(-aq,1,q)$ is a cousin of the original Lebesgue identity \eqref{Lebesgue} with an extra $q$-factorial, $1/(q)_n$, in the summand. This extra factor will be the source of the weights in the combinatorial/partition theoretic study of the identities related to the \eqref{Ffunction}. For other references related to weighted partition identities of this spirit one can refer to \cite{Alladi_Weighted, BerkovichUncu3, Uncu}, and in a wider perspective some other recent weighted partition identities can be found in \cite{AlladiPartialTheta, Dixit, BerkovichUncu4}.

Before any combinatorial study, we would like to note the following theorem.
\begin{theorem}\label{Main_theorem_analytic}For variables $a,\ z$ and $q$, we have
\begin{equation}\label{Main_Identity_analytic}
\sum_{n\geq0} \frac{(za)_n(z q^{n+1})_\infty}{(q)_n} z^n q^{\frac{n(n+1)}{2}} = \sum_{n\geq0} \frac{(-za)_n(-zq^{n+1})_\infty}{(q)_n}(-z)^n q^{\frac{n(n+1)}{2}}.
\end{equation}
\end{theorem}

Please note that the only difference between the left- and right-hand sides of \eqref{Main_Identity_analytic} is $z\mapsto-z$. In other words, the object is even in the variable $z$. In author's view, the observed symmetry makes this identity visually highly pleasing.

The following sections are arranged as follows. In Section~\ref{Section_Pf}, we give a proof of Theorem~\ref{Section_Pf} and note some Corollaries of this result. In Section~\ref{Section_Comb}, we study the partition theoretic interpretations of the results in Section~\ref{Section_Pf}.

\section{Proof of Theorem~\ref{Main_theorem_analytic}}\label{Section_Pf}

We require two main ingredients for the proof of \eqref{Main_Identity_analytic}. First, it is a known fact that \begin{align}
\label{limit}\lim_{\rho \rightarrow \infty} \frac{(\rho)_n}{\rho^n} &= (-1)^n q^{\frac{n(n-1)}{2}},
\intertext{and, second, Heine Transformation \cite[p. 241, III.2]{GasperRahman}}
\label{Heine2} \sum_{n\geq 0} \frac{(a)_n(b)_n}{(q)_n(c)_n} z^n &= \frac{(c/b)_\infty(bz)_\infty}{(c)_\infty(z)_\infty} \sum_{n\geq 0} \frac{(abz/c)_n(b)_n}{(q)_n(bz)_n} \left(\frac{c}{b}\right)^n 
\end{align} 

\begin{proof}[Proof of Theorem~\ref{Main_theorem_analytic}] The function $\F(a,z,q)$ can be written as the following due to \eqref{limit}:
\[\F(a,z,q) = \lim_{\rho \rightarrow \infty}\sum_{n\geq0} \frac{(za)_n(\rho)_n}{(q)_n(zq)_n} \left(-\frac{zq}{\rho}\right)^n.\] Then we can directly apply the Heine transformation \eqref{Heine2}, and after tending $\rho\rightarrow\infty$, one gets \begin{equation}
\label{raw_main_identity}\F(a,z,q) = \frac{(-zq)_\infty}{(zq)_\infty} \F(a,-z,q).
\end{equation}
Multiplying both sides of \eqref{raw_main_identity} with $(zq)_\infty$, carrying the infinite $q$-Pochhammers inside the sums, and doing elementary simplifications in the summand level finishes the proof.
\end{proof}

It is evident that some special cases of \eqref{Main_Identity_analytic} (such as $(a,z,q) = (q,1,q)$) can be summed by utilizing simple summation formulas (such as  \cite[II.2, p 354]{GasperRahman} and shown to be equal to $(q^2;q^2)_\infty$). This is not our motivation. We would like to look at special cases of \eqref{Main_Identity_analytic} to extract some combinatorial information. The $(a,z,q) = (q,1,q)$ and $ (-q,1,q)$ cases are presented in Corollary~\ref{Analytic_z_equlas_1_Corollary}.

\begin{corollary}\label{Analytic_z_equlas_1_Corollary} Let $q$ be a variable, we have
\begin{align}
\label{cor_1_eqn_1} \sum_{n\geq 0} (q^{n+1})_\infty q^{\frac{n(n+1)}{2}} &= \sum_{n\geq 0} (-q^{n+1})_\infty \frac{(-q)_n}{(q)_n} (-1)^n q^{\frac{n(n+1)}{2}},\\
\label{cor_1_eqn_2} \sum_{n\geq 0} (-q^{n+1})_\infty (-1)^n q^{\frac{n(n+1)}{2}} &= \sum_{n\geq 0} (q^{n+1})_\infty \frac{(-q)_n}{(q)_n} q^{\frac{n(n+1)}{2}}.
\end{align}
\end{corollary}

Another interesting corollary can be seen by picking $a=z=1$ in \eqref{raw_main_identity} and using Jacobi Triple Product identity \cite[p. 239, II.2]{GasperRahman}, \begin{equation}
\label{JTP} \sum_{n=-\infty}^\infty z^n q^{n^2} = (-zq;q^2)_\infty (-q/z;q^2)_\infty (q^2;q^2)_\infty.
\end{equation}

\begin{corollary}\label{JTP_Corollary} We have
\[ \sum_{n\geq 1} (-1)^n q^{n^2} = \sum_{n\geq 1} \frac{(-1)^n q^{{\frac{n(n+1)}{2}}}}{(q)_n (1+q^n)}.\]
\end{corollary}

\begin{proof} It is clear that only the $n=0$ term of the sum on the left-hand side of \eqref{raw_main_identity} is non-zero when $a=z=1$, and the total sum on the left hand side is 1: \[1 = \frac{(-q)_\infty}{(q)_\infty} \sum_{n\geq 0}\frac{(-1)_n }{(q)_n(-q)_n} (-1)^n q^{\frac{n(n+1)}{2}}.\] We multiply both sides of this equation by $(q)_\infty / (-q)_\infty$ and observe that \[\frac{(q)_\infty}{(-q)_\infty} = \frac{(q;q^2)_\infty(q^2;q^2)_\infty}{(-q)_\infty} =\frac{(q;q^2)_\infty (q)_\infty (-q)_\infty}{(-q)_\infty} = (q;q^2)^2_\infty (q^2;q^2)_\infty.\] The right-hand side of the last line is the same as the right-hand side of \eqref{JTP} with $z=-1$. This yields \begin{equation}
\sum_{n=-\infty}^\infty (-1)^n q^{n^2} = \sum_{n\geq 0}\frac{(-1)_n }{(q)_n(-q)_n} (-1)^n q^{\frac{n(n+1)}{2}},
\end{equation}where the left-hand side is coming from \eqref{JTP} and the right-hand side is $\F(1,-1,q)$. Splitting the bilateral sum on the left-hand side and using simple cancellations on $(-1)_n / (-q)_n$, using the definition of the $q$-factorials, on the right-hand side, we get \begin{equation}
1+ 2 \sum_{n\geq 1}(-1)^n q^{n^2} = 1 + 2 \sum_{n\geq 1} \frac{(-1)^n q^{{\frac{n(n+1)}{2}}}}{(q)_n (1+q^n)}.
\end{equation} This shows claim.
\end{proof}

Another proof of this result appears in the author's joint paper with Berkovich as Lemma~4.1 \cite{BerkovichUncu4}. Combinatorial interpretation of this identity was done by Bessenrodt--Pak \cite{BessenrodtPak} and later by Alladi \cite{AlladiPartialTheta} studies the combinatorial implications of this identity.

\section{Partition Theoretic Interpretations of Corollary~\ref{Analytic_z_equlas_1_Corollary}}\label{Section_Comb}

We would like to interpret the identities \eqref{cor_1_eqn_1} and \eqref{cor_1_eqn_2} as weighted partition identities. To that end, we need to define what a partition is and some related statistics. A \textit{partition} (in \textit{frequency notation} \cite{Theory_of_Partitions}) is a list of the form \[(1^{f_1},2^{f_2},3^{f_3},\dots )\] where $f_i \in \mathbb{N}\cup \{0\}$ and all but finitely many $f_i$ are non-zero. When writing example partitions down, one tends to drop the zero frequency parts to keep the notation clean. 

If none of the \textit{frequencies} $f_i$ are greater than 1, we call these partitions \textit{distinct}. One can define the \textit{size} of a partition $\pi$ as \[|\pi| = \sum_{i\geq 1} i\cdot f_i, \] and the sum of all $f_i$ is the number of parts in a partition, we denote this by $\#(\pi)$. The partition with $f_i\equiv 0$ for all $i\in\mathbb{N}$ is the only partition of 0 with 0 parts. 

Let $t(\pi)$ be the number of non-zero frequencies of a partition $\pi$ starting from $f_1$. In other words, one can think of $t(\pi)$ as the length of the initial frequency chain. The length of the initial frequency chain seems to be an underutilized statistic in interpretations of $q$-series identities, the only other closely related statistic that the author knows of is used in \cite[Thm 3.1]{BerkovichUncu4}. Let $p_j(\pi)$ be the maximum index $i$ such that $f_i \geq j$ in $\pi$ and for all $k\geq i$ has the property $f_k < j$, if no positive value satisfies this we define $p_j(\pi)=0$. Let $r_j(\pi)$ be the number of different parts with frequencies $\geq j$.

To exemplify the statistics defined, let $\pi=(1^4,2^2,3^4,5^1,6^1)$ then $|\pi| = 31$, $\#(\pi) = 12$, $t(\pi) = 3$, $p_1(\pi) = 6$, $p_2(\pi)=3$, $p_3(\pi)=3$, $p_4(\pi)=3$, $p_5(\pi)=0$, $\dots$, $r_1(\pi) = 5$, $r_2(\pi)=3$, $r_3(\pi)=2$, $r_4(\pi)=2$, $r_5(\pi)=0$, $\dots$.

With the statistics defined above, one can interpret Corollary~\ref{Analytic_z_equlas_1_Corollary} as a weighted partition theorem, where $i=1$ corresponds to \eqref{cor_1_eqn_1} and $i=2$ refers to \eqref{cor_1_eqn_2},  as follows.

\begin{theorem}\label{Comb_Thm} Let $\D$ be the set of distinct partitions and let $\A$ be the set of partitions where all the partitions $\pi\in\A$ satisfy $p_2(\pi)\leq t(\pi)$. Then for $i=1$ and 2, we have
\begin{equation}
\label{weighted abstract} \sum_{\pi\in\D} w_i(\pi) q^{|\pi|}  =\sum_{\pi\in \A} \hat{w}_i(\pi) q^{|\pi|},
\end{equation}where
\begin{align}
\label{w_i} w_i(\pi) &= \left[1- f_1\, \left(\frac{1- (-1)^{t(\pi)}}{2}\right)  \right](-1)^{i \#(\pi)}, \\
\label{hat_w_i} \hat{w}_i(\pi) &=2^{r_2(\pi)}\left(\frac{(-1)^{t(\pi)}+(-1)^{p_2(\pi)}}{2} \right)(-1)^{(i-1) (r_1(\pi)+t(\pi)+p_2(\pi))}.  
\end{align}
\end{theorem}

We would like to exemplify Theorem~\ref{Comb_Thm} with relevant partitions of 6 in Table~\ref{Table_1}.

\begin{table}[h]\caption{Partitions of 6 from $\D$ and $\A$ and the related weights $w_i$ and $\hat{w}_i$ to exemplify \ref{weighted abstract}.}\label{Table_1}\vspace{-0.8cm}
\[\begin{array}{cccc||ccccccc}
\pi \in\D 	& t(\pi)& w_1(\pi)	& w_2(\pi) 	& \pi\in\A	& t(\pi)	& p_2(\pi)	&r_2(\pi)	&	\hat{w}_1 &r_1(\pi)	&\hat{w}_2\\ \hline
(6^1)		& 	0	&  -1		&	1		& (6^1)	 	&	0		&	0		&	0		&	1		& 1	&	-1			\\
(1^1,5^1)	&	1	&	0		&	0		& (1^1,5^1)	&	1		&	0		&	0		&	0		& 2	&	0			\\
(2^1,4^1)	&	0	&	1		&	1		& (2^1,4^1)	&	0		&	0		&	0		&	1		& 2	&	1			\\
(1^1,2^1,3^1)&	3	&	0		&	0		& (1^2,4^1)	&	1		&	1		&	1		&	-2		& 2	&	-2			\\
			&		&			&			& (1^1,2^1,3^1)	&	3	&	0		&	0		&	0		& 3	&	0			\\
			&		&			&			& (1^3,3^1)	&	1		&	1		&	1		&	-2		& 2	&	-2			\\
			&		&			&			& (1^4,2^1)	&	2		&	1		&	1		&	0		& 2	&	0			\\
			&		&			&			& (1^2,2^2)	&	2		&	2		&	2		&	4		& 2	&	4			\\ 
			&		&			&			& (1^6)		&	1		&	1		&	1		&	-2		& 1	& 	2			\\ \hline
Total : 	&		&	0		&	2		&			&			&			&			&	0		&	&	2			
\end{array}\]
\end{table}

One key observation is that $w_2(\pi) = |w_1(\pi)| \geq 0$ for all distinct partitions. This proves that the series in \eqref{cor_1_eqn_2}, which is the analytic version of \eqref{weighted abstract} with $i=2$, have non-negative coefficients. We write this as a theorem using an equivalent form of the left-hand side series of \eqref{cor_1_eqn_2}.

\begin{theorem}\label{Thm_positivity}We have
\[(-q;q)_\infty \sum_{n\geq 0} \frac{(-1)^n q^{\frac{n(n+1)}{2}}}{(-q,q)_n} \succcurlyeq 0, \] where $\succcurlyeq 0$ is used to indicate that the series coefficients are all greater or equal than 0.
\end{theorem}

The sum in Theorem~\ref{Thm_positivity} is a false theta function that Rogers studied \cite{Rogers}. Although this series has alternating signs, its product with the manifestly positive factor $(-q;q)_\infty$ has non-negative coefficients and the above key observation is a combinatorial explanation of this fact. 

Broadly speaking, connections of false/partial theta functions and their implications in the theory of partitions have been studied in various places. Interested readers can refer to \cite{BerkovichUncu3,AlladiPartialTheta}.

\begin{proof}[Proof of Theorem~\ref{Comb_Thm}] This theorem is a consequence of Corollary~\ref{Analytic_z_equlas_1_Corollary}, the $i=1$ and $2$ cases correspond to the combinatorial interpretations of \eqref{cor_1_eqn_1} and \eqref{cor_1_eqn_2}, respectively. 

First we focus on the left-hand side summands. For a fixed $n$ and $\varepsilon_1 = \pm 1$, $(\varepsilon_1 q^{n+1})_\infty$ is the generating function for the distinct partitions $\pi_d$ where every part is $\geq n+1$ counted with the weight $(-\varepsilon_1)^{\#(\pi_d)}$. We also interpret the $q$-factor, $\varepsilon_2^n q^{n(n+1)/2}$ as the partition $\pi_i = (1^1,2^1,\dots,n^1)$ counted with the weight $\varepsilon_2^n$, where $\varepsilon_2=\pm 1$. We can combine (add the frequencies of both partitions) $\pi_d$ and $\pi_i$ into a distinct partition $\pi$. 

In the sum, \[\sum_{n\geq 0} (\varepsilon_1 q^{n+1})_\infty\varepsilon_2^n q^{\frac{n(n+1)}{2}}, \] there are $t(\pi)+1$ possible pairs $(\pi_d,\pi_i)$ that can yield $\pi$, and one needs to count the weights of these accordingly. Note that if $t(\pi)\geq 1$ since $\pi$ is a distinct partition $f_1=1$. For the total weight of $\pi$, one needs to sum from $k=0$ to $t(\pi)$ of the alternating weights $(-\varepsilon_1)^{\#(\pi)-k}\varepsilon_2^k$: \[\sum_{k = 0}^{t(\pi)} (-\varepsilon_1)^{\#(\pi)-k}\varepsilon_2^k.\] By reducing the summations of alternating weights, one finds that $w_i(\pi)$ can be represented as in \eqref{w_i} for $i=1$ and 2, where $\varepsilon_1=\varepsilon_2=1$ and $\varepsilon_1=\varepsilon_2=-1$, respectively.

Similar to the left-hand side's interpretation, one needs to look at the pieces of the right-hand side summand. For a fixed $n$, once again the parts $(\varepsilon_1 q^{n+1})_\infty\varepsilon_2^n q^{n(n+1)/2}$ can be interpreted as the generating function for the partition pairs $(\pi_d,\pi_i)$ counted by some weights dependent of $\varepsilon_1$ and $\varepsilon_2$. The new factor $(-q)_n/(q)_n$ is the generating function for the number of overpartitions $\pi_o$, into parts $\leq n$. Overpartitions are the same as partitions counted with the weight $2^{r_1(\pi)}$. When we combine $\pi_d$, $\pi_i$ and $\pi_o$, we end up with a partition $\pi$ where some parts may repeat. 

Any repetition of the parts in $\pi$ comes from the overpartition $\pi_o$ and these repetitions can only appear for part $\leq t(\pi)$. Note that $\pi_i$, has a single copy of every part size up to $t(\pi)$ and $\pi_o$ may add more occurrences of these parts. This modifies the overpartition related weight a little and we need to take the first occurrence of a part for granted. On the other hand, if a part appears more than once the repetition should be counted with the weight $2^{r_2(\pi)}$. 

Here the summation bounds are slightly different than the previous case. One needs to sum all the possible $\varepsilon_1$ and $\varepsilon_2$ related weights from $k=p_2(\pi)$ to $t(\pi)$. Different than the previous one, $\#(\pi)$ is replaced by the number of non-repeating parts above the initial chain $t(\pi)$, which is $r_1(\pi)-t(\pi)$. Moreover, one needs to replace $k$ by $k-p_2(\pi)$ to eliminate the effect of the parity of $p_2(\pi)$ on the alternating sum. Hence, the sum to reduce here is \[\sum_{k = p_2(\pi)}^{t(\pi)} (-\varepsilon_1)^{r_1(\pi)-t(\pi)-p_2(\pi)-k}\varepsilon_2^k.\] These sums, once reduced, can be seen to yield $\hat{w}_i(\pi)$ for $i=1$ and $2$, where $\varepsilon_1=\varepsilon_2=-1$ and $\varepsilon_1=\varepsilon_2=1$, respectively.
\end{proof}

\section{Acknowledgement}

The author would like to thank the SFB50-09 and SFB50-11 Projects of the Austrian Science Fund FWF for supporting his research.

\end{document}